\documentclass{amsart} 
\usepackage{amscd,amssymb,latexsym,subfigure,verbatim,epsfig,amsmath,supertabular,xypic}
\usepackage[graph,frame,poly,arc]{xy}
\usepackage{calc}
\usepackage{amsmath}
\usepackage{graphicx}
\begin{document}

\newcommand{\mmbox}[1]{\mbox{${#1}$}}
\newcommand{\proj}[1]{\mmbox{{\mathbb P}^{#1}}}
\newcommand{\affine}[1]{\mmbox{{\mathbb A}^{#1}}}
\newcommand{\Ann}[1]{\mmbox{{\rm Ann}({#1})}}
\newcommand{\caps}[3]{\mmbox{{#1}_{#2} \cap \ldots \cap {#1}_{#3}}}
\newcommand{\N}{{\mathbb N}}
\newcommand{\Q}{{\mathbb Q}}
\newcommand{\Z}{{\mathbb Z}}
\newcommand{\R}{{\mathbb R}}
\newcommand{\K}{{\mathbb K}}
\newcommand{\p}{{\mathbb P}}
\newcommand{\pp}{{\mathfrak p}}
\newcommand{\J}{{\mathcal J}}
\newcommand{\RJ}{{\mathcal R}/{\mathcal J}}
\newcommand{\RI}{{\mathcal R}/{\mathcal I}}
\newcommand{\CC}{{\mathcal C}}
\newcommand{\C}{{\mathbb C}}
\newcommand{\CR}{S^r(\hat \Delta)}
\newcommand{\CRP}{S^r(\hat P)}
\newcommand{\stv}{\mathop{\rm st(v)}\nolimits}
\newcommand{\rja}{\bigoplus\limits_{\alpha\in\Delta^0_{i+1}}{\mathcal{R}}/{\mathcal{J}}(\alpha)}
\newcommand{\rjc}{\bigoplus\limits_{\gamma\in\Delta^0_{i-1}}{\mathcal{R}}/{\mathcal{J}}(\gamma)}
\newcommand{\rjb}{\bigoplus\limits_{\beta\in\Delta^0_i}{\mathcal{R}}/{\mathcal{J}}(\beta)}
\newcommand{\rjlp}{\bigoplus\limits_{\gamma\in\Delta^0_{i-1}}{\mathcal{R}}/{\mathcal{J}}(\gamma)_\pp}
\newcommand{\rjtp}{\bigoplus\limits_{\beta\in\Delta^0_i}{\mathcal{R}}/{\mathcal{J}}(\beta)_\pp}
\newcommand{\hiprjp}{H_{i+1}({\mathcal{R}}/{\mathcal{J}})_\pp}
\newcommand{\hirjp}{H_i({\mathcal{R}}/{\mathcal{J}})_\pp}
\newcommand{\hirj}{H_i({\mathcal{R}}/{\mathcal{J}})}
\newcommand{\Tor}{\mathop{\rm Tor}\nolimits}
\newcommand{\Ass}{\mathop{\rm Ass}\nolimits}
\newcommand{\ann}{\mathop{\rm ann}\nolimits}
\newcommand{\Ext}{\mathop{\rm Ext}\nolimits}
\newcommand{\Hom}{\mathop{\rm Hom}\nolimits}
\newcommand{\im}{\mathop{\rm im}\nolimits}
\newcommand{\rank}{\mathop{\rm rank}\nolimits}
\newcommand{\supp}{\mathop{\rm supp}\nolimits}
\newcommand{\CB}{Cayley-Bacharach}
\newcommand{\coker}{\mathop{\rm coker}\nolimits}
\sloppy
\newtheorem{defn0}{Definition}[section]
\newtheorem{prop0}[defn0]{Proposition}
\newtheorem{conj0}[defn0]{Conjecture}
\newtheorem{thm0}[defn0]{Theorem}
\newtheorem{lem0}[defn0]{Lemma}
\newtheorem{corollary0}[defn0]{Corollary}
\newtheorem{example0}[defn0]{Example}

\newenvironment{defn}{\begin{defn0}}{\end{defn0}}
\newenvironment{prop}{\begin{prop0}}{\end{prop0}}
\newenvironment{conj}{\begin{conj0}}{\end{conj0}}
\newenvironment{thm}{\begin{thm0}}{\end{thm0}}
\newenvironment{lem}{\begin{lem0}}{\end{lem0}}
\newenvironment{cor}{\begin{corollary0}}{\end{corollary0}}
\newenvironment{exm}{\begin{example0}\rm}{{$\Diamond$}\end{example0}}

\newcommand{\defref}[1]{Definition~\ref{#1}}
\newcommand{\propref}[1]{Proposition~\ref{#1}}
\newcommand{\thmref}[1]{Theorem~\ref{#1}}
\newcommand{\lemref}[1]{Lemma~\ref{#1}}
\newcommand{\corref}[1]{Corollary~\ref{#1}}
\newcommand{\exref}[1]{Example~\ref{#1}}
\newcommand{\secref}[1]{Section~\ref{#1}}
\newcommand{\poina}{\pi({\mathcal A}, t)}
\newcommand{\poinc}{\pi({\mathcal C}, t)}
\newcommand{\std}{Gr\"{o}bner}
\newcommand{\jq}{J_{Q}}

\title {Algebraic methods in approximation theory}
\author{Hal Schenck}
\thanks{Schenck supported by  NSF 1312071}\address{Schenck: Mathematics Department \\ University of Illinois Urbana-Champaign\\
  Urbana \\ IL 61801\\ USA}
\email{schenck@math.uiuc.edu}


\begin{abstract}
\noindent This survey gives an overview of several fundamental algebraic 
constructions which arise in the study of splines. Splines play a key role in 
approximation theory, geometric modeling, and numerical analysis; their properties 
depend on combinatorics, topology, and geometry of a simplicial or polyhedral subdivision
of a region in $\R^k$, and are often quite subtle. 
We describe four algebraic techniques which are useful in the study of
splines: homology, graded algebra, localization, and inverse systems. Our goal is
to give a hands-on introduction to the methods, and illustrate them with 
concrete examples in the context of splines. We highlight progress
made with these methods, such as a formula for the third
coefficient of the polynomial giving the dimension of the spline space in high degree. The objects appearing here may be computed using the {\tt spline} package of the {\tt Macaulay2} software system.
\end{abstract}
\maketitle
\tableofcontents
\section{Introduction}\label{sec:intro}
 In mathematics it is often useful to approximate a function $f$ on
a region by a simpler function. A natural way to do this is to
divide the region into simplices or polyhedra, and then approximate $f$ on each
simplex by a polynomial function. A $C^r$-differentiable 
piecewise polynomial function on a $k$-dimensional simplicial or polyhedral 
subdivision $\Delta \subseteq \mathbb{R}^k$ is called a {\em spline}. 
Splines are ubiquitous in geometric modeling and 
approximation theory, and play a key role in the finite element 
method for solving PDE's. There is also a great
deal of beautiful mathematical structure to these problems,
involving commutative and homological algebra, geometry, 
combinatorics and topology. 

For a fixed $\Delta$ and choice of smoothness $r$, the set 
of splines where each polynomial has degree
at most $d$ is a real vector space, denoted
$S^r_d(\Delta)$. The dimension of $S^r_d(\Delta)$ depends 
on $r$,$d$ and the geometric, combinatorial, and topological 
properties of $\Delta$. For many important cases, there
is no explicit general formula known for this dimension. 
In applications, it will also be important to find a good basis,
or at least a good generating set for $S^r_d(\Delta)$; in this
context good typically means splines which have a small support set. 

Splines seem to have first appeared in a paper of Courant \cite{c}, 
who considered the $C^0$ case. Pioneering work by Schumaker \cite{schumaker1} 
in the planar setting established a dimension formula for all
$d$ when $\Delta$ has a unique interior vertex, as well as 
a lower bound for any $\Delta$: 
\begin{thm}$[$Schumaker, \cite{schumaker1}$]$\label{Sformula} For a simplicial complex $\Delta \subseteq \mathbb{R}^2$
\[\dim S^r_d(\Delta) \ge {d+2 \choose 2} + {d-r+1 \choose 2}f^0_1 - \left( {d+2 \choose 2} - {r+2 \choose 2}\right)f^0_0 + \sum \sigma_i
\]
\noindent where $f_1^0 =| \mbox{interior edges}|$, $f_0^0 =| \mbox{interior vertices}|$, and $\sigma_i = \sum_j \max\{(r+1+j(1-n(v_i))), 0 \}$, with $n(v_i)$ the number of distinct slopes at an interior vertex $v_i$.
\end{thm}
Using Bezier-Bernstein techniques, Alfeld-Schumaker prove in \cite{as1} that if $d\ge 4r+1$ then equality holds in Theorem~\ref{Sformula}, Hong \cite{h} shows equality holds if $d \ge 3r+2$, and Alfeld-Schumaker show in \cite{as2} equality holds for $d\ge 3r+1$ and 
$\Delta$ generic. There remain tantalizing open questions in the 
planar case: the Oberwolfach problem book from 
May 1997 contains a conjecture of Alfeld-Manni that for $r=1$ 
Theorem~\ref{Sformula} gives the dimension in degree $d =  3$. 
Work of Diener \cite{d} and 
Tohaneanu \cite{t} shows the next conjecture is optimal.
\begin{conj}\label{2r1} \cite{sthesis}
The Schumaker formula holds with equality if $d \ge 2r+1$.
\end{conj}
Homological methods were introduced to the field in a watershed 
1988 paper of Billera \cite{b}, which solved a conjecture of  
Strang \cite{strang} on the dimension of $S^1_2(\Delta)$ for
a generic planar triangulation. One key ingredient in the work
was a result of Whiteley \cite{w} using rigidity
theory. Homological methods are discussed
in detail in \S 2, and the utility of these tools is illustrated
in \S 4. 
 
A useful observation is that the smoothness condition is local: 
for two $k$ simplices $\sigma_1$ and $\sigma_2$ sharing a common $k-1$ 
face $\tau$, let $l_\tau$ be a nonzero linear form vanishing on $\tau$. 
Then a pair of polynomials $f_1, f_2$ meet with order $r$ smoothness 
across $\tau$ iff $l_\tau^{r+1} | f_1-f_2. $ 
For splines on a line, the situation is easy to understand, so the
history of the subject really begins with the planar case. Even the
simplest case is quite interesting: let $\Delta \subseteq \R^2$ be
the star of a vertex, so that $\Delta$ is triangulated with a 
single interior vertex as in the next example.

\pagebreak

\begin{exm}\label{fexm} A planar $\Delta$ which is the star of a single interior vertex $v_0$ at the origin.
\begin{center}
\epsfig{figure=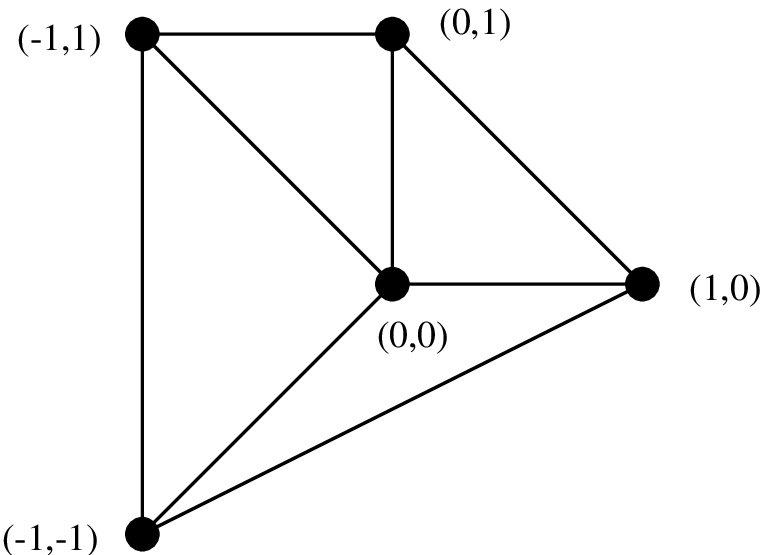,height=1.5in,width=2.0in}
\end{center}

Starting with the triangle in the first quadrant and moving clockwise, label 
the polynomials on the triangles $f_1,\ldots, f_4$. To obtain a
global $C^r$ function, we require
\begin{center}
\begin{equation}\label{firstE}
\begin{array}{ccc}
a_1 y^{r+1}&=&f_1-f_2\\
a_2 (x-y)^{r+1}&=&f_2-f_3\\
a_3 (x+y)^{r+1}&=&f_3-f_4\\
a_4 x^{r+1}&=&f_4-f_1
\end{array}
\end{equation}
\end{center}
Summing both sides yields the equation $\sum_{i=1}^4a_i l_i^{r+1} = 0$ (where $l_1=y$ and so on), and gives a hint that algebra has a role to play.
\end{exm}
\begin{defn}
Let $\{f_1, \ldots, f_m\}$ be a set of polynomials.
A syzygy is a relation 
\[
\sum\limits_{i=1}^m a_if_i = 0, \mbox{ where the }a_i \mbox{ are also polynomials}.
\]
\end{defn}
Notice that if each $f_i$ is a fixed polynomial $f$, then the smoothness 
condition is trivially satisfied. 
\begin{defn} 
For any $r,d, \Delta$, the set of polynomials of degree at most $d$ is a subspace of 
$S^r_d(\Delta)$, which we call global polynomials.  
\end{defn} 
Using the vector space structure and the global polynomials, we may
assume $f_1 =0$. This means that given a syzygy on 
\[
\{ y^{r+1}, (x-y)^{r+1}, (x+y)^{r+1}, x^{r+1} \},
\] 
we can reverse the process and solve for the $f_i$. 
So when $\Delta =\stv \subseteq \R^2$ is triangulated with a single interior vertex as in the example above, 
$S^r_d(\stv)$ consists of global polynomials and syzygies. To actually compute the dimension and basis 
for the space of syzygies is nontrivial:
\begin{thm}$[$Schumaker, \cite{schumaker1}$]$
For a planar simplicial complex $\Delta = \stv$ with $f^0_1$ interior edges and $n$ distinct slopes at the interior vertex $v$, 
\[
\begin{array}{ccc}
\dim S^r_d(\stv) &= &{d+2 \choose 2}\\
                & + &{d-r+1 \choose 2}f^0_1\\                                                     & - &\left( {d+2 \choose 2} - {r+2 \choose 2}\right)\\     
                & + &\sum\limits_{j\ge 0} \max\{(r+1+j(1-n)), 0 \}
\end{array}
\]
\end{thm}
Throughout this paper, $\Delta \subseteq \R^k$ is a simplicial or polyhedral complex; $\Delta_i$ and $\Delta_i^0$ denote the sets of $i$-dimensional faces and $i$-dimensional interior faces; $f_i(\Delta) = |\Delta_i|$ and $f_i^0(\Delta)=|\Delta_i^0|$. Finally, $\Delta$ is a pseudomanifold: every $k-1$ simplex $\tau \not \in \partial(\Delta)$ is a face of exactly two $k$-simplices, and for any pair $\sigma, \sigma' \in \Delta_k$, there is a sequence $\sigma=\sigma_1, \ldots \sigma_m=\sigma'$ with $\sigma_i \cap \sigma_{i+1} \in \Delta_{k-1}$. Our references are de Boor~\cite{dB} and Lai-Schumaker~\cite{ls} for splines and Eisenbud~\cite{e} and Schenck~\cite{sBook} for algebra. 
\section{Homology and chain complexes}
In this section, we give an overview of the methods introduced
in the study of splines by Billera in \cite{b}. Subsequent progress using homological methods appears in Billera-Rose \cite{br1}, \cite{br2}, Yuzvinsky \cite{y}, Schenck-Stillman \cite{ss97a}, \cite{ss97b}, Rose \cite{r1},\cite{r2}, Mourrain-Villamizar \cite{mv}, DiPasquale \cite{mdp1}, \cite{mdp2}, and more. We begin with a review of the algebraic and topological basics, and then specialize to the case of splines.
\subsection{Algebraic setting}
A sequence of vector spaces and linear transformations
$$\CC:\mbox{ }\cdots{\stackrel{\phi_{j+2}}{\longrightarrow}} V_{j+1} {\stackrel{\phi_{j+1}}{\longrightarrow}} V_{j} 
{\stackrel{\phi_{j}}{\longrightarrow}} V_{j-1} 
{\stackrel{\phi_{j-1}}{\longrightarrow}} \cdots$$
is called a complex (or chain complex) if $$\im(\phi_{j+1}) \subseteq \ker(\phi_{j}).$$
The sequence is exact at position $j$ if 
$\im(\phi_{j+1}) = \ker(\phi_{j})$; a complex
which is exact everywhere is called an exact sequence. We define 
the homology of the complex $\CC$ as 
$$H_j(\CC) = \ker(\phi_{j})/ \im(\phi_{j+1}).$$
\begin{exm}\label{Hfirst}
Consider the complex $$ 0 \longrightarrow V_1 
{\stackrel{\phi}{\longrightarrow}} V_0 \longrightarrow 0,$$
where $V_1=V_0=\R^3$ and $\phi$ is:
$$
\left[ \begin{array}{ccc}
-1&0&1\\
1&-1&0\\
0 &1&-1
\end{array} \right] $$
$H_1(\CC) = \ker(\phi)$ has basis $[1,1,1]^t$ and 
$H_0(\CC) = \coker(\phi) = \R^3/\im(\phi)$. 
\end{exm}
For a complex of finite dimensional vector spaces
\[
\CC: 0 \longrightarrow V_n \longrightarrow V_{n-1} \longrightarrow \cdots
\longrightarrow V_1\longrightarrow V_0 \longrightarrow 0
\]
the alternating sum of the dimensions is called the 
Euler characteristic of $\CC$, and written $\chi(\CC)$; when $\CC$ is exact 
$\chi(\CC)=0$, and an easy induction shows that in general 
\begin{equation}\label{Euler}
\chi(\CC) = \sum_{i=0}^n (-1)^i \dim V_i =\sum_{i=0}^n (-1)^i \dim H_i(\CC).
\end{equation}
Everything defined above generalizes in the obvious way to sequences
of modules and homomorphisms. 
\begin{defn}\label{sescpx}
A short exact sequence of complexes is a commuting diagram:
\[
\xymatrix{ 
&  0 \ar[d] & 0\ar[d]   & 0 \ar[d] \\
A: \cdots  \ar[r]^{\partial_3}& A_2 \ar[r]^{\partial_2}  \ar[d] & A_1 \ar[r]^{\partial_1}  
\ar[d]  & A_0\ar[r]  \ar[d] & 0\\
B: \cdots   \ar[r]^{\partial_3} &B_2 \ar[r]^{\partial_2}  \ar[d] & B_1 \ar[r]^{\partial_1}  
\ar[d]  & B_0\ar[r]  \ar[d] & 0\\
C: \cdots   \ar[r]^{\partial_3} &C_2 \ar[r]^{\partial_2}  \ar[d] & C_1 \ar[r]^{\partial_1}  
\ar[d]  & C_0\ar[r]  \ar[d] & 0\\
 &  0 & 0  & 0 
}
\]
where the columns are exact and the rows are complexes.
\end{defn}
\begin{thm}\label{FTHA}
A short exact sequence of complexes yields a long
exact sequence in homology:
\[
\xymatrixcolsep{16pt}
\xymatrix{
\cdots \ar[r] & H_{n+1}(C)  \ar[r] &  H_{n}(A)  \ar[r] &  H_{n}(B)
 \ar[r] & H_{n}(C) \ar[r] &H_{n-1}(A)  \ar[r] & \cdots
}
\]
\end{thm}
For additional details and proof, see Chapter 8 of \cite{sBook}. 
\subsection{Topological motivation}
The invention of homology has topological roots. The basic idea is to
model topological objects with combinatorial ones; in particular there
is no topological difference between the circle and the boundary of
a triangle, or between the sphere and the boundary of a tetrahedron. 
This generalizes, and the idea is to approximate a topological 
space $X$ by using simplices, which are higher dimensional analogs of
triangles. A union of simplices forms a simplicial complex $\Delta$, 
giving a combinatorial approximation to $X$. The data of $\Delta$
allows us to build an algebraic chain complex ${\mathcal C}(\Delta)$,
whose homology encodes topological information.

An abstract $n$-simplex is a set consisting of all subsets of an 
$n+1$ element ground set. Typically a simplex is viewed as a geometric
object; for example a two-simplex on the set $\{a,b,c\}$ can be visualized
as a triangle, with the subset $\{a,b,c\}$ corresponding to the whole 
triangle, $\{a,b\}$ an edge, and $\{a\}$ a vertex. 
\begin{defn}$[$Chapter 5, \cite{sBook}$]$
A simplicial complex $\Delta$ on a vertex set $V$ is a collection of subsets
$\gamma$ of $V$, such that if $\gamma \in \Delta$ and $\tau \subseteq \gamma$, 
then $\tau \in \Delta$. If $|\gamma| = i+1$ then $\gamma$ is called an $i-$face.
An oriented simplex is a simplex with a fixed ordering of the vertices,
modulo an equivalence relation: for a permutation $\sigma \in S_n$ and
oriented simplex $\tau =[i_1,\ldots, i_n]$, $\tau \sim (-1)^{\mbox{sgn}(\sigma)} \sigma(\tau)$.
\end{defn}
\begin{exm}\label{orient1}
Label the vertices in Example~\ref{fexm} as $v_0 = (0,0), v_1 = (1,0), v_2 = (-1,-1), v_3 = (-1,1), v_4 = (0,1)$. Then the set of four oriented triangles
\[
[v_0,v_1,v_4], [v_0,v_2,v_1], [v_0,v_3,v_2],[v_0,v_4,v_3]
\]
eight oriented edges
\[
[v_0,v_1],[v_0,v_2],[v_0,v_3],[v_0,v_4],[v_1,v_2],[v_2,v_3],[v_3,v_4],[v_4,v_1]
\]
and vertices $[v_0],\ldots, [v_4]$ form an oriented simplicial complex. 
\end{exm}

\begin{defn}\label{BDMP}
Let $C_i(\Delta)$ be a free $R$-module with basis the oriented $i$-simplices, 
and define a map $C_i(\Delta)\stackrel{\partial_i}{\longrightarrow}C_{i-1}(\Delta)$ 
via 
\[
\partial_i[e_{j_0},\ldots e_{j_i}] =\sum\limits_{m=0}^i (-1)^m  [e_{j_0},\ldots \widehat{e_{j_m}}, \ldots e_{j_i}]
\]
\end{defn}
A check shows that $\partial_i \circ \partial_{i+1} =0$, hence $(C(\Delta), \partial)$ is a chain complex. The homology of $(C(\Delta), \partial)$ encodes topological information. If $\Delta = \{[v_0],[v_1],[v_2],[v_0v_1],[v_1v_2],[v_2v_0]\}$,
then $\Delta \sim S^1$ and $C(\Delta)$ is the complex of Example~\ref{Hfirst}.
$H_1(\Delta) \sim \R$ captures the fact that $S^1$ is not simply connected. 
\subsection{Splines and homology}
The compatibility condition discussed in \S 1 has a beautiful interpretation
in terms of homology: suppose $\Delta \subseteq \R^k$ and 
$\sigma$ and $\sigma' \in \Delta_k$ satisfy 
\[
\sigma \cap \sigma' = \tau \in \Delta_{k-1}.
\]
Then if $f$ is a polynomial on $\sigma$ and $f'$ a polynomial on $\sigma'$,
the set of pairs $(f,f')$ which glue $C^r$ smoothly across $\tau$ is the 
kernel of the map
\begin{equation}\label{smoothF}
R^2 \stackrel{[1,-1]}{\longrightarrow}R/l_{\tau}^{r+1}.
\end{equation}
\begin{exm}
In {\em relative} homology, the modules $C_i(\Delta)$ are quotiented
by chains $C_i(\Delta')$ of a subcomplex $\Delta'$. In the spline setting, 
the subcomplex is $\partial(\Delta)$. For Example~\ref{orient1}, applying
Definition~\ref{BDMP} for the relative complex yields
\[
\partial_2 =\left[ \begin{array}{cccc}
1&-1&0&0\\
0&1&-1&0\\
0&0 &1&-1\\
-1&0&0 &1
\end{array} \right] 
\]
We compute 
\[
\partial_2[f_1,f_2,f_3,f_4]^t =\left[ \begin{array}{c}
f_1-f_2\\
f_2-f_3\\
f_3-f_4\\
f_4-f_1
\end{array} \right] 
\]
This is the right hand side of Equation~\ref{firstE}. 
\end{exm}
We still need to encode the smoothness condition, and Equation~\ref{smoothF} provides the clue:
rather than having $\partial_2$ map a free module to another free module, we enrich our chain complex to include the smoothness condition:
\begin{exm}\label{Alltogether}
Continuing with the previous example, define a map
\[
\bigoplus\limits_{\sigma \in \Delta_2}R \stackrel{\partial_2}{\longrightarrow}\bigoplus\limits_{\tau \in \Delta_1^0}R/l_{\tau}^{r+1}.
\]
The kernel of this map consists exactly of polynomial vectors $[f_1,f_2,f_3,f_4]$
which satisfy the conditions of Equation~\ref{firstE}. This suggests how to 
define a chain complex whose top homology module consists of splines 
on $\Delta$.
\end{exm}

\begin{exm}\label{Constcomplex}
Let ${\mathcal{R}}$ be the constant complex on
$\Delta^0$: ${\mathcal{R}}(\sigma) = R$, for every $\sigma \in
\Delta^0$. Take $\partial_i$ to be the usual (relative to $\partial\Delta$) simplicial boundary map. Then $H_i(\mathcal R)$ is the usual (modulo boundary) simplicial homology, with coefficients in $R$.
\end{exm}

\begin{exm}\label{Billcomplex}
In \cite{b}, Billera defined the following complex: 
for each $\sigma \in \Delta^0$, let $I_\sigma$ be the 
ideal of $\sigma \subset \R^{k}$.  $I_{\sigma}$ is
generated by linear polynomials. Fix $r \in \N$, and 
define a complex ${\mathcal I}$ of ideals on $\Delta$ by 
${\mathcal I}(\sigma) = I^{r+1}_\sigma$, and define the quotient 
complex ${\mathcal{R}}/{\mathcal I}$ via 
${\mathcal{R}}/{\mathcal I}(\sigma) = R/I^{r+1}_\sigma$.
\end{exm}

\begin{exm}\label{HMcomplex}
The paper \cite{ss97a} modifies Billera's complex.
Fix $r \in \N$, and define a complex ${\mathcal J}$ of ideals 
on $\Delta$ by 
\[
J_\psi  = \langle l_{\tau_1}^{r+1}, \ldots, l_{\tau_n}^{r+1} \mid \psi
\in \tau_i \in \Delta^0_{d-1} \rangle.
\]
Define the quotient complex ${\mathcal{R}}/{\mathcal J}$ via 
${\mathcal{R}}/{\mathcal J}(\psi) = R/J_\psi$.
\end{exm}
While the complexes $\RI$ and $\RJ$ agree at positions $k$ and $k-1$,
they typically differ in lower degrees. Since $\mathcal J$ is
a submodule of $\mathcal R$, the differential in the complex of 
Example~\ref{Constcomplex} induces a differential on $\mathcal J$ and
on $\RJ$. 
By Theorem~\ref{FTHA}, the short exact sequence of complexes
\begin{equation}\label{rjses}
0\longrightarrow{{\mathcal J}}\longrightarrow{\mathcal{R}}\longrightarrow{\mathcal{R}}/{{\mathcal J}}\longrightarrow 0
\end{equation}
gives rise to a long exact sequence in homology:
\begin{equation}\label{rjles}
\xymatrixcolsep{16pt}
\xymatrix{
 \ar[r] & H_{i+1}(\RJ)  \ar[r] &  H_{i}(\mathcal J)  \ar[r] &  H_{i}(\mathcal R)
 \ar[r] & H_{i}(\RJ) \ar[r] &H_{i-1}(\mathcal J)  \ar[r] & \cdots
}
\end{equation}
Billera showed $H_k(\RI)_d = S^r_d(\Delta)$. Since ${\mathcal J}$ and 
${\mathcal I}$ agree on the $k$ and $k-1$ faces, 
$$H_k(\RJ)_d= H_k(\RI)_d = S^r_d(\Delta).$$ 
However, the lower homology modules differ; in the complex $\RJ$,
information is more evenly balanced between the homology modules and 
the modules of the chain complex: neither of these sets of modules is 
simple to understand. On the other hand, the modules in the chain complex 
$\RI$ are easy to understand, so that for the chain complex $\RI$, 
all geometric information is encoded in the homology modules, making them 
difficult to decipher. In \S 4, we use localization to prove \newline
\begin{thm}[\cite{s1}]~\label{HPdim}
If $\Delta \subseteq \R^k$ simplicial, then for all $i < k$ and $d\gg 0$, $\dim_{\R}H_i(\RJ)_d$ is given by a polynomial in $d$ of degree at most $i-2$.
\end{thm}
For $\Delta \subseteq \R^k$ and $d \gg 0$, 
the dimension of $S^r_d(\Delta)$ is given by a polynomial of 
degree $k$, and a corollary of Theorem~\ref{HPdim} is that a suitable 
analog of the Schumaker formula gives the top three coefficients of the polynomial, for any $k$. 

\section{Graded algebra}\label{GA}
From our earlier discussion, it follows that 
we can think of a spline as a vector with polynomial
entries, one polynomial for each maximal face, satisfying certain conditions 
ensuring that the resulting function (a priori only defined on individual faces)
is globally a $C^r$ smooth function. Furthermore, we can add splines of
the same order of smoothness. Given a $C^r$ spline represented as a vector of polynomials $(f_1,\ldots,f_n)$, multiplying the vector by a fixed 
polynomial $f$ gives $(f \cdot f_1,\ldots,f \cdot f_n)$, which is again 
a $C^r$ spline. This means that the set of splines is more than
just a vector space; it is a module over the polynomial ring. In this
section we build on this extra structure. Roughly speaking, the
main advantage to this approach is that it packages all the vector
spaces $S^r_d(\Delta)$ into a single object, where tools of algebra
and algebraic geometry can be brought to bear on the problem.
\subsection{Rings and modules, Hilbert polynomial and series}
Let $R$ denote the polynomial ring $\R[x_1,\ldots,
x_{k+1}]$, in the setting of splines $k$ is the dimension of the ambient
space.  The polynomial ring has a special structure not shared by arbitrary
rings: it has a grading by $\Z$:
\[
R = \bigoplus_{i \in \Z}R_i,
\]
where $R_i$ denotes the set of homogeneous (each monomial is of the same degree) 
polynomials of degree $i$, and if $r_i \in R_i$ and $r_j \in R_j$,
then $r_i \cdot r_j \in R_{i+j}$. A graded $R$-module $M$ is defined in 
similar fashion. Of special interest
is the case where $R_0$ is a field, for then each $M_i$ is 
a vector space. An $R$-module $M$ is free if it is isomorphic to $R^m$ for 
some $m \in \N$; a free rank one $R$-module with generator in degree $i$ is denoted 
$R(-i)$, so $R(i)_j = R_{i+j}$. 
\begin{exm}\label{exm:shift}
Let $R=\R[x,y]$. The table below gives a bases for $R_i$ and $R(-2)_i$
\begin{center}
\begin{supertabular}{|c|c|c|} 
\hline $i$ & $R_i$ &  $R(-2)_i$  \\ 
\hline $0$ & $1$ & $0$\\
\hline $1$ & $x,y$ &$0$\\ 
\hline $2$ & $(x, y)^2$ &$1$ \\ 
\hline $3$ & $(x,y)^3$ & $x,y$ \\ 
\hline $4$ & $(x,y)^4$ & $(x,y)^2$ \\
\hline \vdots &\vdots & \vdots \\
\hline 
\end{supertabular}
\end{center}
\end{exm}
\begin{defn} The  Hilbert function $HF(M,d) = \dim_{\R}M_d.$ \end{defn}
\begin{defn} The Hilbert series $HS(M,t) = \sum_{\Z}\dim_{\R}M_it^i.$ \end{defn}
Induction shows that $HS(R(-i),t)=t^i/(1-t)^{\ell}$ and $HF(R(-i),d) =
{d+k-i \choose d}$ if $d \ge i$. For a finitely generated graded
$R$-module $M$, the Hilbert function becomes polynomial for $d\gg 0$
(\cite{sBook}, Theorem 2.3.3), and is denoted $HP(M,d)$. 
\begin{exm}\label{exm:first} 
For an example of a non-free $R$-module, let $R=\R[x,y]$,  
and $M=R/\langle x^2,xy \rangle$. The table below gives a bases for $M_i$ and $M(-2)_i$  
\begin{center}
\begin{supertabular}{|c|c|c|} 
\hline $i$ & $M_i$ &  $M(-2)_i$  \\ 
\hline $0$ & $1$ & $0$\\
\hline $1$ & $x,y$ &$0$\\ 
\hline $2$ & $y^2$ &$1$ \\ 
\hline $3$ & $y^3$ & $x,y$ \\ 
\hline $4$ & $y^4$ & $y^2$ \\
\hline \vdots &\vdots & \vdots \\
\hline  
\end{supertabular}
\end{center}
\noindent The respective Hilbert series are  
\[ 
HS(M,i)=\frac{1-2t^2+t^3}{(1-t)^2} \mbox{ and}
\] 
\[
HS(M(-2),i)=\frac{t^2(1-2t^2+t^3)}{(1-t)^2}
\] 
\end{exm}
\subsection{Free resolutions}
It is easy to compute the Hilbert series and Hilbert polynomial of a graded 
module from a finite free resolution. 
\begin{defn}A finite free resolution for an $R$-module $M$ is an exact sequence 
\[
\mathbb{F} : 0 \rightarrow F_{k+1} \rightarrow \cdots \rightarrow F_i \stackrel{d_i}{\rightarrow} F_{i-1}\rightarrow \cdots \rightarrow F_0  \rightarrow M \rightarrow 0,
\]
where the $F_i$ are free $R$-modules; $\mathbb{F}$ exists by the Hilbert syzygy theorem \cite{e}.
\end{defn}
\begin{exm}
For $R/\langle x^2,xy\rangle$, a finite free resolution is
\[
0 \longrightarrow R(-3) \xrightarrow{\left[ \!
\begin{array}{c}
y \\
-x 
\end{array}\! \right]} R(-2)^2
\xrightarrow{\left[ \!\begin{array}{cc}
x^2& xy
\end{array}\! \right]}
 R \longrightarrow R/I \longrightarrow 0.
\]
\begin{center}
The map $[x^2,xy]$ sends 
$\begin{array}{c}
\mbox{    }e_1 \mapsto x^2\\
\mbox{    }e_2 \mapsto xy.
\end{array}$
\end{center}
In order to have a map of graded modules, the basis elements of the
source must have degree two, explaining the shifts in the free
resolution. Here is where graded objects are useful: 
looking at a single fixed degree, we obtain an exact sequence
of vector spaces; by Equation~\ref{Euler} the alternating sum of these
dimensions is the dimension of $M_d$. For example, the alternating 
sum of the Hilbert series of the free modules above gives 
\[
HS(M,i)=\frac{t^3-2t^2+1}{(1-t)^2}
\]
which agrees with our earlier computation.
\end{exm}

\subsection{Grading in the spline setting}\label{hatspline}
Billera and Rose observed in \cite{br1} that if $\hat \Delta$ is 
the simplicial complex obtained by embedding $\Delta$ in the 
plane $\{z_{k+1}=1\} \subseteq \mathbb{R}^{k+1}$
and forming the cone with the origin, then the set of 
splines (of all degrees) on $\hat \Delta$ is a graded module 
$S^r(\hat\Delta)$ over $\R[x_1,\ldots,x_{k+1}]$, and 
$S^r(\hat\Delta)_d \simeq S^r_d(\Delta).$ 

In algebraic terms, we want the Hilbert series of $S^r(\hat\Delta)$;
if only asymptotic information (the dimension for $d\gg 0$) is
needed, then it suffices to compute the Hilbert polynomial 
of $S^r(\hat\Delta)$.
\begin{lem}\label{lem:BR}$[$Billera-Rose, \cite{br1}$]$
Let $P \subseteq \R^k$ be a $k$--dimensional polyhedral complex. Then
there is a graded exact sequence:
\[
0\longrightarrow \CRP \longrightarrow R^{f_k}\oplus
R^{f_{k-1}^0}(-r-1) \stackrel{\phi}{\longrightarrow}
R^{f_{k-1}^0}
\longrightarrow N \longrightarrow
0
\]
\begin{equation}\label{BRmatrix}
 \hbox{where  } \phi = \;\;{\small \left[ \partial_k \Biggm| \begin{array}{*{3}c}
l_{\tau_1}^{r+1} & \  & \  \\
\ & \ddots & \  \\
\ & \ & l_{\tau_m}^{r+1}
\end{array} \right]}
\end{equation}
\end{lem}
Write $[\partial_k \mid D]$ for $\phi$. To describe $\partial_k$, note 
that the rows of $\partial_k$ are indexed by $\tau \in P_{k-1}^0$. 
If $\sigma_1, \sigma_2$ denote the $k-$faces adjacent to $\tau$, 
then in the row corresponding to
$\tau$ the smoothness condition means that the only nonzero entries
occur in the columns corresponding to $\sigma_1, \sigma_2$, and are
$\pm(+1,-1)$. When $P$ is simplicial, 
$\partial_k$ is the top boundary map in the (relative) chain complex. 
Billera-Rose show $N$ is supported on 
primes of codimension at least two, which allows them to 
determine the top two coefficients of the Hilbert polynomial for
arbitrary $k$; a refinement of this appears in Alfeld~\cite{a}. 
Both leave open the question of the $O(d^{k-2})$ terms of $HP(\CR,d)$,
which are determined in \cite{s1} and \cite{s2}.
\section{Localization}
A key concept in many areas of mathematics is that of a quotient of an
object $G$ by some subset $H$; the quotient $G/H$ is typically a
simpler object. An analogous but less familiar operation is
localization; rather than zeroing out $H$ as in the quotient
construction, localization makes elements in the subset invertible.
\subsection{Basics of the construction}
\begin{defn}
Let $R$ be a ring and $S$ a multiplicatively closed subset of $R$
containing $1$. Define an equivalence relation on $\{\frac{r}{s} \mid
r \in R, s \in S\}$ via
\[
\frac{r_1}{s_1} \sim \frac{r_2}{s_2} \mbox{ if } (r_1s_2-r_2s_1)s_3 =
0 \mbox{ for some } s_3 \in S
\]
The localization $R_S$ is the set of equivalence classes; it is easily
checked to be a ring.
\end{defn}
The most common usages are when $S$ is either the set of all
multiples of some element $\{1, r, r^2, \cdots\}$, or when $S$ is 
the complement of a prime ideal $\pp$. Recall that an ideal $\pp$ is prime if
$ab \in \pp$ implies either $a \in \pp$ or $b \in \pp$; this condition means 
the complement is indeed multiplicatively closed. In particular, the
process of localizing at a prime ideal $\pp$ makes every element outside
$\pp$ invertible. 
\begin{exm}
In $\Z$, $\langle 0 \rangle$ is prime, so $S = \Z \setminus \{0\}$, and 
$\Z$ localized at $\langle 0 \rangle$ is $\Q$. 
\end{exm}
Given a module $M$ over a ring $R$ and prime ideal $\pp \subseteq R$, 
then the localization $M_\pp$ is constructed as above; $M_\pp$ is an 
$R_\pp$-module, and is isomorphic to $M \otimes_R R_\pp$. 
\begin{exm}
Let $R=\R[x,y,z]$ and $M = R/\langle xy, xz \rangle$. Localizing $M$
and $R$ at the prime $\pp = \langle x \rangle$, we see that 
$$M_\pp \simeq R_\pp/\langle x \rangle R_\pp,$$ 
and localizing at the prime $\pp' = \langle y, z \rangle$ we find that $M_{\pp'} \simeq  R_{\pp'}/\langle y,z \rangle
R_{\pp'}$. So localization really does allow us to focus in on local
properties.
\end{exm}
\begin{defn}\label{AssPrimes}
A prime ideal $\pp$ is associated to a graded $R$-module $M$ if $\pp$ is
the annihilator of some $m \in M$. $\Ass(M)$ 
denotes the set  of associated primes of $M$. 
\end{defn}
The reason that the associated primes of a module $M$ are important is that
they are exactly the prime ideals at which the localization $M_\pp \ne 0$. 
A key fact we will need (\cite{sBook}, Theorem 6.1.3) is that 
localization preserves exact sequences.
\subsection{Application: polyhedral splines}
In \cite{schumaker2}, Schumaker obtained upper and lower bounds for splines on a 
planar polyhedral complex $P$. We sketch results of \cite{ms} yielding an analog
of Theorem~\ref{Sformula} in the polyhedral setting. 
The top two coefficients of the Hilbert polynomial depend on $f_2(P),
f_1^0(P)$ and $r$, and agree with the Schumaker formula for the
simplicial case. However, the constant term differs, and in a very
interesting fashion. In addition to Schumaker's work on bounds, in \cite{y}, Yuzvinsky 
obtains results on freeness in the polyhedral setting. 
\begin{exm}\label{TH}
Let $P$ be the polygonal complex depicted below, and $P'$ be a complex obtained by perturbing (any) vertex so that the affine spans of the 
three edges which connect boundary vertices to interior vertices are not concurrent. 
\begin{figure}[h]
\begin{center}
\epsfig{file=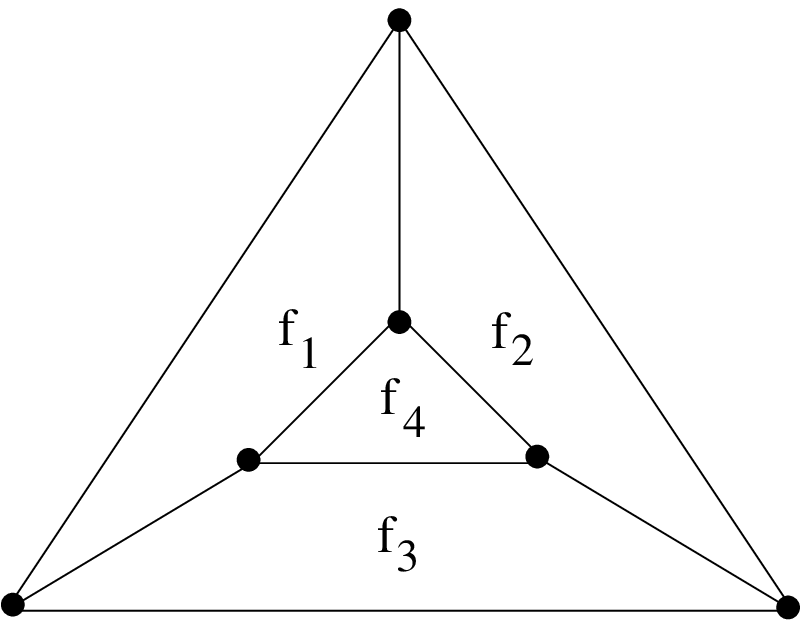,height=1.3in,width=1.5in}
\end{center}
\end{figure}
\noindent The following table gives the dimensions for $S^r_d(P)$ and
$S^r_d(P')$ for small values of $r$, as long as $d \gg 0$. 
\vskip .1in
\begin{center}
\begin{supertabular}{|c|c|c|}
\hline $r$ & $\dim_{\mathbb{R}} S^r_d(P)$ & $\dim_{\mathbb{R}} S^r_d(P')$ \\
\hline $0$ & $2d^2+2$ & $2d^2+1$ \\ 
\hline $1$ & $2d^2-6d+10$ & $2d^2-6d+7$\\ 
\hline $2$ & $2d^2-12d+32$ & $2d^2-12d+25$\\ 
\hline $3$ & $2d^2-18d+64$ & $2d^2-18d+52$\\ 
\hline $4$ & $2d^2-24d+110$ & $2d^2-24d+91$\\ 
\hline
\end{supertabular}
\end{center}
Theorem~\ref{PlanarMain} gives a complete explanation of this table.
\end{exm}
We sketch the strategy: it 
follows from additivity of the Hilbert polynomial on exact sequences
and Lemma \ref{lem:BR} that obtaining
the coefficient of $d^{k-2}$ in the Hilbert polynomial of $\CR$
is equivalent to obtaining the coefficient of $d^{k-2}$ in the 
Hilbert polynomial of $N$. Since
\[
N \simeq (\!\!\!\bigoplus\limits_{\tau \in P_{k-1}^0} R/l_{\tau}^{r+1})/\im(\partial_k),
\]
every element of $N$ is annihilated by some $r \in R$. 
Using localization, we first
show that the codimension two associated primes of $N$ must be
{\em linear}, then give a precise description of which codimension two linear primes actually occur. This leads to an 
explicit description of the submodule of $N$ supported in codimension two. 
Elements of this submodule are the only elements of $N$ which contribute
to the $d^{k-2}$ coefficient of the Hilbert polynomial, and the
formula follows.
\begin{exm}\label{TH2}
In Example~\ref{TH}, label the boundary vertices as $1,2,3$,
starting from the top vertex and moving clockwise, and the interior
vertices $4,5,6$ in the same way. Choose as an oriented basis for
$\Delta_1^0$ 
\[
\{[41],[52],[63],[45],[56],[64]\}
\]
and for 
$\Delta_2$ 
\[
\{[1364],[1452],[2563],[465]\}
\]
Then the matrix for $\partial_2$ is 
\[
{\small \left[ \begin{array}{*{4}c}
1 & -1 & 0 & 0\\
0 &  1 & -1 & 0 \\
-1 &  0& 1 & 0\\
0& 1 & 0&-1 \\
0&0&1 & -1 \\
1 & 0 &0&-1 
\end{array} \right],}
\]
and so the matrix $\phi$ of Equation~\ref{BRmatrix} is 
\[
\phi = \;\;{\small \left[ \begin{array}{*{10}c}
1 & -1 & 0 & 0 & l_1^{r+1} & 0 & 0 & 0&0&0\\
0 &  1 & -1 & 0  &0&l_2^{r+1} & 0 & 0 & 0&0\\
-1 &  0& 1 & 0  & 0&0&l_3^{r+1} & 0 & 0 & 0\\
0& 1 & 0&-1  &0&0&0&l_4^{r+1} & 0 & 0 \\
0&0&1 & -1 &  0 & 0 & 0&0 &l_5^{r+1} & 0 \\
1 & 0 &0&-1  & 0 & 0 & 0&0&0& l_6^{r+1} 
\end{array} \right].}
\]
Let $p$ be the point where the affine spans of edges $[14],[25],[36]$
meet, and let $\pp$ be the ideal of polynomials vanishing at $p$. So
$\{l_1,l_2,l_3\} \subseteq \pp$ and $l_4, l_5, l_6$ are not in $\pp$. This
means in the localization at $\pp$, the forms $l_4, l_5, l_6$ are units.
Since $N_\pp$ is the cokernel of $\phi_\pp$, because $l_4, l_5, l_6$ are 
units, the images in $N_\pp$ of the last three rows are zero. In
particular, $N_\pp$ is the cokernel of
\[
\phi_\pp = \;\;{\small \left[ \begin{array}{*{10}c}
1 & -1 & 0 & 0 & l_1^{r+1} & 0 & 0 & 0&0&0\\
0 &  1 & -1 & 0  &0&l_2^{r+1} & 0 & 0 & 0&0\\
-1 &  0& 1 & 0  & 0&0&l_3^{r+1} & 0 & 0 & 0
\end{array} \right],}
\]
$N_\pp$ has three generators, but quotienting by the first three columns of $\phi_\pp$ makes the three generators equal in the cokernel; in particular
\[
N_\pp \simeq R_\pp/\langle  l_1^{r+1} , l_2^{r+1}, l_3^{r+1} \rangle
\]
Note that if we perturb a vertex so that the three lines are not
concurrent, then one of $l_1, l_2,l_3$ will become a unit in $R_\pp$. 
This forces $N_\pp$ to vanish, and explains the difference between $S^r_d(P)$ 
and $S^r_d(P')$.
\end{exm}
\begin{lem}\label{lem:linear1}
Any prime ideal $\pp$ associated to $N$ contains a
linear form $l_\tau$, for some $\tau \in P^0_{k-1}$.
\end{lem}
\begin{proof}
From the description 
\[
N \simeq \big( \!\!\!\bigoplus\limits_{\tau \in P_{k-1}^0} R \big/ l_{\tau}^{r+1} \big) \big/ \im(\partial_k),
\]
it follows that if no $l_\tau$ is in $\pp$, then all the $l_\tau$ are
invertible in $R_\pp$, so that $N_\pp$ vanishes.
\end{proof}

\begin{lem}\label{lem:main2}
Let $\xi = V(\pp)$ be a linear space. 
If $\sigma \in P_k$ has at most one facet whose linear span contains $\xi$, 
then every generator of $N$ corresponding to a facet of 
$\sigma$ is mapped to zero in the localization $N_{\pp}$.
\end{lem}
\begin{proof}
In $R_{\pp}$, any $l_\tau$ with $\xi \not\subseteq V(l_\tau)$ becomes 
invertible. As $N$ is the cokernel of 
\[
\phi = \;\;{\small \left[ \partial_k \Biggm| \begin{array}{*{3}c}
l_{\tau_1}^{r+1} & \  & \  \\
\ & \ddots & \  \\
\ & \ & l_{\tau_m}^{r+1}
\end{array} \right],}
\]
in the right hand diagonal submatrix $D$ of $\phi$, all the forms 
(or all save one) $l_\tau^{r+1}$ such that 
$\tau$ is a facet of $\sigma$ become units. As the column 
of the left hand ($\partial_k$) matrix corresponding to $\sigma$ 
has nonzero entries only in rows corresponding to facets of $\sigma$,
this means that every generator corresponding to a facet 
of $\sigma$ has zero image in the localization, and the 
result follows.
\end{proof}
\begin{thm}\label{thm:linear2}
Any codimension two prime ideal $\pp$ associated to $N$ is of the form 
$\langle l_{\tau_1},l_{\tau_2} \rangle$ for $\tau_i \in P_{d-1}^0$
such that $V(l_{\tau_1}, l_{\tau_2})$ has codimension two.
\end{thm}
\begin{proof}
If there do not exist two $l_i$ as above, then by Lemma
\ref{lem:linear1}, $V(\pp)$ is contained in exactly one hyperplane which is the linear span of $\tau \in
P_{k-1}^0$. Thus, in $R_\pp$, all but one of the $l_\tau$ become units,
and the proof of Lemma \ref{lem:main2} shows that $N_\pp$ vanishes.
\end{proof}

Theorem~\ref{thm:linear2} gives an explicit set of candidates for the
codimension two primes of $N$. By Billera and Rose \cite{br1} 
all associated primes of $N$ have codimension at least two 
(this also follows from the argument above), so the theorem identifies 
all candidates for the associated primes of minimal codimension.
In order to determine exactly which codimension two linear primes are
actually associated to $N$, we introduce a graph, which
depends on both combinatorics and geometry of $P$. In
the simplicial case, the codimension two associated primes are 
exactly the vertices of $P$. In the polyhedral
case, the geometry is more subtle.

\begin{defn}\label{defn:dualG2}
Let $P$ be a $k$--dimensional polyhedral complex embedded in
$\mathbb{R}^k$, and $\xi$ a
codimension two linear subspace. $G_\xi(P)$ is a graph
whose vertices correspond to those $\sigma \in P_k$ such that 
there exists a $(k-1)$--face of $\sigma$ whose linear span contains
$\xi$. Two vertices of $G_\xi(P)$ are joined iff
the corresponding $k$--faces share a common $(k-1)$--face whose linear
span contains $\xi$.
\end{defn}
\begin{exm}\label{TH3}
In Example~\ref{TH}, each interior vertex $v$ of $P$ has $G_v(P)$ a triangle. 
Additionally, if $\xi$ is the point at which affine spans of the three edges connecting interior vertices to boundary vertices meet, then $G_\xi(P)$ is as below, where $v_i$ corresponds to the facet labeled $f_i$ in the figure of Example~\ref{TH}.

\begin{figure}[h]
\begin{center}
\epsfig{file=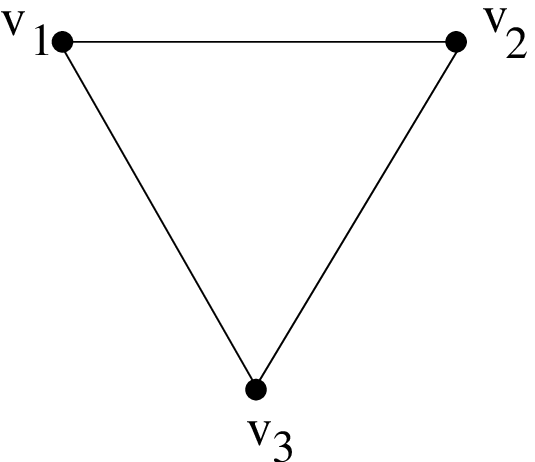,height=1.2in,width=1.4in}
\end{center}
\end{figure}

Let $P'$ be obtained by moving the top vertex of $P$ a bit to the
right. Then $P$ and $P'$ are combinatorially equivalent, but in $P'$
there are no sets of $\ge 3$ concurrent $V(l_\tau), \tau \in
P^0_{d-1}$, except at the interior vertices. In particular, $G_\xi(P')$ is acyclic. 
\end{exm}

\begin{lem}\label{lem:main1}
For any $\sigma \in P_k$, there are at most two facets of $\sigma$ 
whose linear spans contain a given codimension two linear space $\xi$. 
\end{lem}
\begin{proof}
Suppose the linear spans of three facets $\tau_1,\tau_2,\tau_3$ of $\sigma$
meet in a codimension two linear space $\xi$. 
For each $V(l_{\tau_i})$, $\sigma$ lies on one 
side of the hyperplane; so $\sigma$ lies between $V(l_{\tau_1})$ 
and $V(l_{\tau_2})$. Since $V(l_{\tau_3})$ contains 
\[
\xi =V(l_{\tau_1}) \cap V(l_{\tau_2}),
\]
this means 
$V(l_{\tau_3})$ would split $\sigma$, a contradiction.
\end{proof}
\begin{cor}\label{cor:cor1}
$G_{\xi}(P)$ is homotopic to a disjoint union of circles and segments.
\end{cor}
\begin{proof}
By Lemma \ref{lem:main1}, the valence of any 
vertex $v \in G_{\xi}(P)$ is at most two.
\end{proof}

\begin{thm}\label{thm:main3}
For a polyhedral $P$ and codimension two linear space $\xi = V(\pp)$,
\[
N_{\pp} \simeq \!\!\!\!\!\! \bigoplus\limits_{\psi \in H_1(G_{\xi}(P))}\!\!\!\!\!\! (R/I_\psi)_{\pp}
\]
where $\psi \in H_1(G_{\xi}(P))$ means $\psi$ is a component of
$G_{\xi}(P)$ homotopic to $S^1$, and 
\[
I_\psi = \langle l_\tau^{r+1} \mid \tau \in P^0_{k-1}
\mbox{ corresponds to an edge of } \psi \rangle.
\]
\end{thm}
\begin{proof}
By Corollary \ref{cor:cor1}, $G_{\xi}(P)$ consists
of a disjoint union of 
cycles and segments. By Lemma \ref{lem:main2}, all generators of
$N$ which lie in a segment are mapped to zero in the localization $N_{\pp}$.
For each $k-$face $\sigma$ corresponding to a vertex in a cycle, note that
there are two $(k-1)-$faces $\tau_1$, $\tau_2$ of $\sigma$ such that 
$l_{\tau_1},l_{\tau_2}$ are not units in $R_{\pp}$; every other
linear form defining a facet of $\sigma$ becomes a unit. Reducing the
column of $\partial_k$ corresponding to $\sigma$ by the columns of $D_{\pp}$
having a unit entry gives a column with nonzero entries only in rows
corresponding to $\tau_1$ and $\tau_2$. Repeating the process shows
that the cycle corresponds to a principal submodule of $N_{\pp}$, 
with the generator quotiented by the $(r+1)^{st}$ powers of the forms
corresponding to the edges of the cycle.
\end{proof}

\begin{prop}\label{prop:ses}
Let $\mathcal{P}$ be the set of all codimension two associated primes 
of $N$. Then there is an exact sequence
\[
0 \longrightarrow K \longrightarrow N  \longrightarrow  \!\!\! \!\!\!  \bigoplus\limits_{\stackrel{\psi \in H_1(G_{V(Q)}(P))}{Q \in \mathcal{P}}}  \!\!\! \!\!\! R/I_\psi \longrightarrow C \longrightarrow 0,
\]
where $K$ and $C$ are supported in codimension at least three.
\end{prop}
\begin{proof}
The reasoning in the proof of Theorem~\ref{thm:main3} shows that if 
$\xi = V(Q)$ with $Q \in \mathcal{P}$, then 
\[
\bigoplus\limits_{\psi \in H_1(G_{V(Q)}(P))} \!\!\!\!\!\! R/I_\psi
\]
is exactly
the cokernel of the submatrix of $[\partial_k \mid D ]$ obtained by
deleting those rows indexed by $\tau \in P_{k-1}^0$ such that 
$\xi \not \in \mbox{conv}(\tau)$. An application of the snake lemma (\cite{sBook}, Lemma 8.1.1) then shows that 
\[
N \longrightarrow \bigoplus\limits_{\psi \in H_1(G_{\xi}(P))} \!\!\!\!\! R/I_\psi  \longrightarrow 0.
\]
is exact. Taking the sum of such maps over all $Q \in \mathcal{P}$ yields
the exact sequence of the proposition. Theorem~\ref{thm:main3} shows
that upon localizing this sequence at any prime $Q \in \mathcal{P}$, 
the localizations $C_{Q}$ and $K_{Q}$ vanish, hence $K$ and $C$ are supported in 
codimension at least three.
\end{proof}

\begin{lem}\label{lem:syzlines}$[$Schumaker, \cite{schumaker1}$]$
Let $I_\psi = \langle l_1^{r+1},\ldots,l_n^{r+1} \rangle \subseteq
\R[x_1,\ldots,x_{k+1}]$ be a codimension two ideal, minimally generated
by the $n$ given elements. Define
\[
\begin{array}{ccc}\alpha(\psi) & =&\lfloor\frac{r+1}{n-1}\rfloor,\\
 s_1(\psi)&=&(n\!-\!1)\alpha(\psi)\!+\!n\!-\!r\!-\!2,\\ 
s_2(\psi)&=&r\!+\!1\!-\!(n\!-\!1)\alpha(\psi).
\end{array}
\]
Then the minimal free resolution of $R/I_\psi$ is:
\[
0  \longrightarrow
\begin{array}{c}
R(-r\!-\!1\!-\!\alpha(\psi))^{s_1(\psi)} \\
\oplus\\
R(-r\!-\!2\!-\!\alpha(\psi))^{s_2(\psi)} 
\end{array}
\!\!\! \longrightarrow
R(-r\!-\!1)^n
 \longrightarrow
R \longrightarrow
R/I_\psi \longrightarrow  0.
\]
\end{lem}
\begin{proof}
See Theorem 3.1 of \cite{ss97b}; the key step involves showing that
a certain matrix has full rank, which was established by Schumaker in
\cite{schumaker2}.
\end{proof}
\noindent It follows from Lemma \ref{lem:syzlines} that 
the Hilbert polynomial of $R/I_\psi$ is given by:

\[
\binom{k\!+\!d}{k}-n\binom{k\!+\!d\!-\!r\!-\!1}{k}  +
s_1(\psi)\binom{k\!+\!d\!-\!r\!-\!1\!-\!\alpha(\psi)}{k} +
s_2(\psi)\binom{k\!+\!d\!-\!r\!-\!2\!-\!\alpha(\psi)}{k}.
\]
\vskip .05in
\begin{thm}\label{PlanarMain}
If $P$ is a hereditary planar polyhedral complex, then 
\[
HP(\CR,d) = \frac{f_2}{2}d^2 + \frac{3f_2-2(r+1)f_1^0}{2}d + f_2+ \Big({r
  \choose 2}-1\Big)f_1^0 + \!\!\! \!\!\! \sum\limits_{\psi_j \in H_1(G_{\xi_i}(P))}\!\!\! c_j,
\]
where
\[
c_j = 1-n(\psi_j){r \choose 2} + s_1(\psi_j){r + \alpha(\psi_j)
  \choose 2} + s_2(\psi_j){r + \alpha(\psi_j)+1 \choose 2} 
\]
\[
= {r+2 \choose 2}+\frac{\alpha(\psi_j)}{2}\Big(2r+3+\alpha(\psi_j)-n(1+\alpha(\psi_j))\Big).
\]
\end{thm}

We close by applying Theorem~\ref{PlanarMain} to Example~\ref{TH}. As we saw in Example~\ref{TH3}, there
are four $\xi$ at which $H_1(G_\xi(P)) \ne 0$, and each $I_\psi$ has three generators. Hence 
the $c_j$ are all the same, and equal to 
\[
{r+2 \choose
  2}+\frac{\alpha(\psi_j)}{2}\Big(2r+3+\alpha(\psi_j)-3(1+\alpha(\psi_j))\Big).
\]
which simplifies to 
\[
 {r+2 \choose 2}+
 \Big \lfloor\frac{r+1}{2}\Big\rfloor(r-\Big\lfloor\frac{r+1}{2}\Big\rfloor).
\]
Theorem~\ref{PlanarMain} yields:
\begin{center}
\begin{supertabular}{|c|c|c|c|c|}
\hline $r$ & $\dim_{\mathbb{R}} S^r_d(P)$ & $ \frac{f_2}{2}d^2 + \frac{3f_2-2(r+1)f_1^0}{2}d$ & $f_2+ \Big({r
  \choose 2}-1\Big)f_1^0$ & $4({r+2 \choose 2}+ \alpha(r-\alpha))$\\
\hline $0$ & $2d^2+2$       & $2d^2$     & $-2$ & $4$\\
\hline $1$ & $2d^2-6d+10$   & $2d^2-6d$  & $-2$ & $12$\\
\hline $2$ & $2d^2-12d+32$  & $2d^2-12d$ & $4 $ & $28$\\
\hline $3$ & $2d^2-18d+64$  & $2d^2-18d$ & $16$ & $48$\\
\hline $4$ & $2d^2-24d+110$ & $2d^2-24d$ & $34$ & $76$\\
\hline
\end{supertabular}
\end{center}
\vskip .4in
For the configuration $P'$ obtained by perturbing a vertex in 
Example \ref{exm:first} so the three edges defining $\xi$ no 
longer meet, there are only three nontrivial $c_j$, hence
\[
\dim_{\mathbb{R}} S^r_d(P')= \dim_{\mathbb{R}} S^r_d(P) - {r+2 \choose 2}- \alpha(r-\alpha)
\]
\subsection{Application: vanishing of homology}
As a second application of localization, we prove Theorem~\ref{HPdim}. Two key properties of localization \cite{e} are that it preserves exactness, and it 
commutes with homology. Write $1_\alpha$ for the unit of $\mathcal{R}/\mathcal{J}(\alpha)$. We have the complex: 
$$ \cdots \longrightarrow \rja \stackrel{\partial_{i+1}}{\longrightarrow} \rjb \stackrel{\partial_i}{\longrightarrow} \rjc \stackrel{\partial_{i-1}}{\longrightarrow} \cdots$$
Let $\pp$ be a prime ideal such that $\mathcal{J}(\gamma) \not\subseteq \pp$, for any $\gamma \in \Delta^0_{i-1}$. Then 
\[
\rjlp = 0, \mbox{ so }\]
\[
\hirjp = \rjtp/(\im(\partial_{i+1}))_\pp.
\]
Now, if $\mathcal{J}(\beta) \not\subseteq \pp$ for any $\beta \in \Delta^0_i$, then 
\[
\rjtp = 0, \mbox{ so }
\]
\[\hirjp = 0,
\]
and we're done. So suppose $\mathcal{J}(\beta)\subseteq \pp$, 
for some (possibly several) $\beta \in \Delta^0_i$. For 
$\alpha \in \Delta^0_{i+1}$, the map $\partial_{i+1}$
takes $1_\alpha$ to a signed sum of $1_\beta$, where $\beta$ is
a facet of $\alpha$. Localization at $\pp$ sends $1_\beta$ to zero if
$\mathcal{J}(\beta) \not\subseteq \pp$. Because $\Delta$ is simplicial, 
two facets of $\alpha$ intersect in a face of dimension $i-1$, and the
assumption that $\pp$ does not contain $\mathcal{J}(\gamma)$ for any
$i-1$ face $\gamma$ implies that in the localization of
$\partial_{i+1}(1_\alpha)$, at most one $1_\beta$ is nonzero, and some $1_\beta$ is nonzero only if $\mathcal{J}(\beta) \subseteq \pp$. Thus $\partial_{{i+1}_\pp}(1_\alpha) = 1_\beta \mbox{ if } \beta \subseteq \alpha \mbox{ and }\mathcal{J}(\beta) \subseteq \pp$, so $\partial_{{i+1}_\pp}$ is surjective and $\hirjp = 0$. 

We have shown that $\hirjp = 0$ if $\pp \not\supseteq
\mathcal{J}(\gamma)$ for any $\gamma \in \Delta^0_{i-1}$. Since $\pp$ is
prime, this means that if $\hirjp \ne 0$, then $\pp \supseteq I(\gamma)$, for some  $\gamma \in \Delta^0_{i-1}$. If we can show that ${H_i({\mathcal{R}}/{\mathcal{J}})_{I(\gamma)}} = 0$ for all $\gamma \in \Delta^0_{i-1}$, then since $I(\gamma)$ is of codimension $k-i+1$, $\hirj$ is supported on primes of codimension at least $k-i+2$, which will conclude the proof. 

Suppose $\pp = I(\gamma)$, some $\gamma \in \Delta^0_{i-1}$. Then the
localized complex splits into a direct sum of subcomplexes, of which
two types can contribute to $\hirjp$. The first type are those of the
following form, with one piece for each $i-1$ face $\gamma_j$ such that $\hat \gamma_j \subseteq V(I(\gamma))$:
$$ \cdots \longrightarrow \bigoplus\limits_{\stackrel{\gamma_j \in \alpha}{\alpha\in\Delta_{i+1}}}{\mathcal{R}/\mathcal{J}(\alpha)}_{I(\gamma)} \longrightarrow \bigoplus\limits_{\stackrel{\gamma_j \in \beta}{\beta\in\Delta^0_i}}{\mathcal{R}}/{\mathcal{J}}(\beta)_{I(\gamma)} \longrightarrow {\mathcal{R}}/{\mathcal{J}}(\gamma_j)_{I(\gamma)} \longrightarrow 0$$

The map $\partial_{i_{I(\gamma)}}$ sends each summand surjectively
to ${\mathcal{R}}/{\mathcal{J}}(\gamma_j)_{I(\gamma)}$, so the
kernel of $\partial_{i_{I(\gamma)}}$ is generated by pairs of units
with opposite orientations (e.g. $1_{\beta_i} - 1_{\beta_j}$), along
with generators of the form
$\l_{\beta_i}^{r+1}\cdot 1_{\beta_j}$, where $\l_{\beta_i}^{r+1} \in
\mathcal{J}(\beta_i)$, but $\l_{\beta_i}^{r+1} \not\in
\mathcal{J}(\beta_j)$. An $i+1$ face $\alpha$ in the above subcomplex has a pair of
$i$ faces $\beta_i$, $\beta_j$ which intersect in $\gamma_j$ (again, we
make use of the fact that $\Delta$ is simplicial), and clearly
$\partial_{{i+1}_{I(\gamma)}}(1_\alpha) = 1_{\beta_i} - 1_{\beta_j}$, which
generate all elements of the kernel of the first type mentioned above. For generators of
the second type, notice that modulo the image
of $\partial_{{i+1}_{I(\gamma)}}$, $\l_{\beta_i}^{r+1}\cdot1_{\beta_j}
= \l_{\beta_i}^{r+1}\cdot1_{\beta_i}$, so is zero in homology. Thus,
these subcomplexes do not contribute to $\hirjp$.

The second type of subcomplex which may contribute to $\hirjp$ are those with $V(I(\gamma)) \subseteq
V(I(\beta_d))$, $\beta_d \in \Delta^0_i$, but where $\beta_d$ does not
contain an $i-1$ face $\gamma_d$ such that $\hat \gamma_d \subseteq V(I(\gamma))$. These complexes
take the form:$$ \cdots \longrightarrow
\bigoplus\limits_{\stackrel{\beta_d \in
    \alpha}{\alpha\in\Delta_{i+1}}}{\mathcal{R}/\mathcal{J}(\alpha)}_{I(\gamma)} \longrightarrow {\mathcal{R}}/{\mathcal{J}}(\beta_d)_{I(\gamma)} \longrightarrow 0$$
It is easy to check that the localized $\partial_{i+1}$ map is
surjective, and hence for these subcomplexes we also have $\hirjp = 0$,
which concludes the proof. Using the technical tool of spectral sequences, one can show more: 

\begin{thm}\label{Free}$[$\cite{s1}$]$
If $\Delta \subseteq \R^k$ is topologically trivial, then 
$\CR$ is a free module if and only if $H_i(\RJ) = 0$ for all $i \le d-1$.
\end{thm}
\vskip .1in

\begin{cor}~\label{corollary:freedim}$[$\cite{s1}$]$
If $\CR$ is free, then the Hilbert series for $\CR$ is determined
entirely by local data $HS(\CR,t)  =  \sum_{i=0}^k (-1)^{k-i}HS(\RJ_i,t)$.
\end{cor}
\begin{proof}
Immediate from Theorem~\ref{Free} and Equation~\ref{Euler}.
\end{proof}
\begin{exm}
Consider the symmetric octahedron pictured below. 
\begin{center}
\epsfig{figure=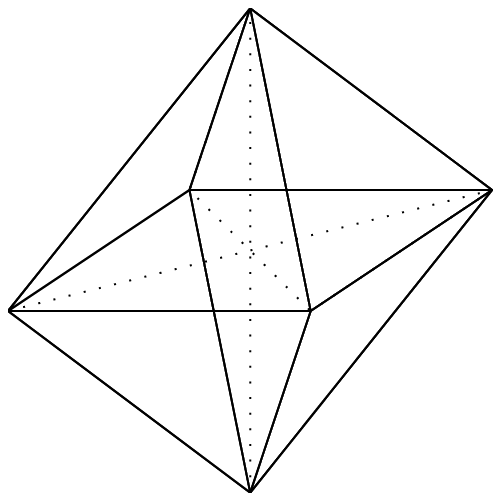,height=1.5in,width=1.5in}
\end{center}
By directly analyzing the maps $\partial_i$ it is possible to
show that $H_2(\RJ) = H_1(\RJ)=0$, hence by Theorem~\ref{Free}, $\CR$ 
is free for all $r$. By Corollary~\ref{corollary:freedim}, the Hilbert
series is given as the alternating sum
\[
HS(\CR,t)  =  \sum\limits_{i=0}^3 (-1)^{3-i}HS(\RJ_i,t) =  \frac{1+3t^{r+1}+3t^{2r+2}+t^{3r+3}}{(1-t)^4}
\]
The computation above is simple because we may translate so that 
the $l_\tau$ are all coordinate
hyperplanes--this example is very nongeneric. It is easy to identify the
generators for $\CR$; for example, in degree $r+1$ there is a generator 
which is $z^{r+1}$ on the top four simplices, and zero on the bottom four; the
other two generators of degree $r+1$ come from symmetry. Continue in this 
fashion. 
\end{exm}
\section{Inverse systems and powers of linear forms}
A century ago, Macaulay \cite{m} defined the notion of an inverse 
system. 
\begin{defn}
Let $S = \R[y_0, \ldots , y_n]$.  We think of $S$ both as a ring, isomorphic 
to $R$, and as an $R$-module where the action $R_i \times S_j \rightarrow 
S_{j-i}$ is that given by partial differentiation. For an ideal $I
\subseteq R$, the inverse system $I^{-1}$ is the set of elements of
$S$ which are annihilated by this action.
\end{defn}
\subsection{Powers of linear forms and fatpoints}
There is a beautiful connection between ideals generated by powers
of linear forms occurring in the complexes $\RJ$ of \S 2, and certain ideals defining zero dimensional objects 
in projective space.
\begin{defn}
Let $p_i = [p_{i0}:p_{i1}:\cdots :p_{in}] \in \p^n$, $I(p_i) = 
P_i \subseteq R = \R[x_0,\ldots, x_n]$. A fatpoints ideal is an ideal of the 
form $I = \cap_{i=1}^m \pp_i^{\alpha_i}$, $\alpha_i \ge 1$. 
\end{defn}
\vskip .1in

\begin{thm}~\label{thm:mainFP} $[$Ensalem and Iarrobino, \cite{ei}$]$ 
For a fatpoints ideal $I = \pp_1^{n_1+1} \cap \cdots
\cap \pp_s^{n_s+1}$, let $l_{p_i} = \sum_{j=0}^np_{i_j}y_j$. Then 
$I^{-1} = ann_S(I)$ may be described as follows:
\[
(I^{-1})_j = \begin{cases} S_j & \mbox{ for }j \leq \max\{n_i\} 
\\                              
       l_{p_1}^{j-n_1}S_{n_1} + \cdots + l_{p_s}^{j-n_s}S_{n_s} & \mbox{ for }j\geq\max\{ n_i + 1\}  
\end{cases}          
\]
and 
\[
\dim_{\R}(I^{-1})_j = \dim_{\R}(R/I, j) 
\]
\end{thm}
Holding $j-n_i = t_i$ fixed and letting $j$ and $n_i$ vary yields 
information about $\langle l_1^{t_1}, \ldots, l_s^{t_s}\rangle$, depending on 
an infinite family of ideals of fatpoints. 
\subsection{Application: planar splines of mixed smoothness}
The paper \cite{gs} uses inverse systems to study splines of
mixed smoothness. First, some preliminaries.
\vskip .1in
\begin{cor}~\label{cor:main1}
Let $l_1, \ldots , l_s$ be $s$ pairwise linearly independent
homogeneous linear forms in 
$S = \R[y_0, y_1]$, $0  < \alpha _1 \leq \cdots \leq \alpha _s$ be
integers, and let  
$J = \langle l_1^{\alpha _1}, \ldots , l_s^{\alpha _s}\rangle$. Then for each 
$t \in \N$, the vector space $J_t$  has the maximum dimension possible:
$$
\dim _d J_t = \min \{ t+1 , \sum _{i=1}^s \max\{ t-\alpha _i + 1, 0 \} \}. 
$$
\end{cor}
\begin{proof}
By \thmref{thm:mainFP}, given an integer $t \ge 0$, 
$$\dim _d J_t = \dim _d(R/I,t)$$
where $$I=\wp_1^{t-\alpha_1+1} \cap \ldots \cap \wp_s^{t-\alpha_s+1}$$
and $\wp_1, \ldots, \wp_s$ are the ideals of the points corresponding to $l_1,\ldots, l_s$ (here we use the convention that $\pp^r = R$ if $r \le 0$). Now $I$ is a principal ideal generated by a form $F$ of degree $d_t$, where 
$$d_t = \sum _{i=1}^s \max\{ t-\alpha _i + 1, 0 \}.$$
So $\dim _d J_t = H(R/I,t) = min (t+1, d_t). $
\end{proof}

\begin{cor}~\label{cor:mingens}
Let $0 < \alpha _1 \leq \alpha _2 \cdots \leq \alpha _t$ and $J = \langle l_1^{\alpha _1}, \ldots , l_t^{\alpha _t}\rangle$. Then if $m \ge 2$:
$$
l_{m+1}^{\alpha _{m+1}} \notin \langle l_1^{\alpha _1}, \ldots , l_m^{\alpha _m} \rangle 
\Leftrightarrow \alpha_{m+1} \le \frac{\sum _{i=1}^m \alpha_i - m}{m-1}.$$
\end{cor}
\begin{proof}
Let $J_m = \langle l_1^{\alpha _1}, \ldots , l_m^{\alpha _m} \rangle$. Then $l_{m+1}^{\alpha _{m+1}} \notin J_m$
if and only if $(J_m)_{\alpha_{m+1}} \ne (J_{m+1})_{\alpha_{m+1}}$. By \corref{cor:main1}, 
$$
\dim (J_m)_{\alpha_{m+1}} = \min \{ \alpha_{m+1}+1 , \sum _{i=1}^m (\alpha_{m+1}-\alpha _i + 1) \},$$
$$
\dim (J_{m+1})_{\alpha_{m+1}} = \min \{ \alpha_{m+1}+1 , \sum _{i=1}^{m+1} (\alpha_{m+1}-\alpha _i + 1) \}. 
$$
Hence, $(J_m)_{\alpha_{m+1}} \ne (J_{m+1})_{\alpha_{m+1}}$ if and only if $$\alpha_{m+1}+1 > \sum _{i=1}^m(\alpha_{m+1}-\alpha_i+1),$$which simplifies to the above condition. 
\end{proof}
From now on, when we write $J = \langle l_1^{\alpha _1}, \ldots ,
l_t^{\alpha _t}\rangle$, we require the exponent vector $\alpha = (\alpha _1, \ldots , \alpha _t)$ of $J$
satisfies the conditions of Corollary~\ref{cor:mingens}, so that we
have a minimal generating set for $J$: for each integer $m \in 2
\ldots t-1$, $\alpha_{m+1} \le \frac{\sum _{i=1}^m \alpha_i - m}{m-1}$.
By \corref{cor:main1}, we also have the following:

\begin{thm}~\label{thm:dim2}
Let $J = \langle l_1^{\alpha _1}, \ldots , l_t^{\alpha _t}\rangle$
with exponent vector of $J$ as above and $d_i$ as in \corref{cor:main1}.  Then 
\[
H(S/J, i) =  \max \{0, i+1-d_i\}. 
\]
\end{thm}
The least integer $\Omega$ for which $H(S/J, \Omega) = 0$ is the least 
integer $p$ such that $p+1-d_p \leq 0$; equivalently $p < \sum_{i=1}^t \max \{p-\alpha _i + 1, 0 \}$.  Thus, $d_{\Omega -1} \le \Omega -1$ and 
$\Omega < d_\Omega$; the socle degree of $S/J$ is $\Omega - 1$.  
Since all the minimal generators of $J$ occur in degree at most one
greater than the socle degree of $S/J$, we see that $\Omega \geq \alpha _i$ for all $i$. 
$$\Omega = \Big\lfloor \frac{\sum_{i=1}^t\alpha_i -t}{t-1}\Big\rfloor + 1.$$
\begin{thm}~\label{thm:res}$[$Geramita-Schenck~\cite{gs}$]$
Let $J$ be an ideal minimally generated by 
$\langle l_1^{\alpha_1}, \ldots , l_t^{\alpha_t}\rangle$, 
so that $\Omega - 1$ is the socle degree of $S/J$. Then $J$ has resolution 
$$
0 \longrightarrow  \begin{array}{c} S(-\Omega-1)^a \\ \oplus\\ S(-\Omega)^{t-1-a} \end{array} \longrightarrow \oplus_{i=1}^t S(-\alpha_i) \longrightarrow J \longrightarrow 0,
$$
\mbox{where }$$
a = H(S/J, \Omega -1) = \sum_{i=1}^t\alpha_i +(1-t)\cdot \Omega.$$
\end{thm}
The proof uses the Hilbert Syzygy Theorem and the Hilbert-Burch
theorem, which may be found in \cite{e}. Theorem~\ref{thm:res} 
generalizes Schumaker's result in Lemma~\ref{lem:syzlines} to allow
varying smoothness. If ${\bf \alpha} \in \N^{f_1^0}$ is an integer 
vector representing the order of smoothness across the interior edges, then this yields 
a simple formula for the Hilbert polynomial of $S^{\bf \alpha}_d(\Delta)$ 
for planar splines of mixed smoothness on $\Delta$.
\subsection{A conjecture in algebraic geometry}
Suppose $\{p_1,\ldots, p_n\}$ is a collection of points in the plane, and
$m_i \in \N$, with $I_X = \cap \pp_i^{m_i}$. A polynomial $f(x,y)$
vanishes
with multiplicity $m_i$ at $p_i$ exactly when all partial derivatives
of $f$ of order $\le m_i-1$ vanish at $p_i$. This places
${m_i-1 \choose 2}$ independent constraints on $f$. If the $p_i$ are
in general position, the natural hope is that the conditions from
distinct points do not interact. Hence if we homogenize the problem, 
our expectation is that the Hilbert function of $R/I_X$ should be
\begin{equation}\label{HFexpected}
HF(R/I_X,j) = {j+2 \choose 2} -\sum_{i=1}^n {m_i+1 \choose 2}
\end{equation}
as soon as $j$ is sufficiently large. 
This naive hope is false:
\begin{exm}
Consider the space of conics through two double points. Then 
\[
{2+2 \choose 2} - 2{2+1 \choose 2} = 6-6 = 0 
\]
so there should be no such conics. But if $V(l)$ is the line
connecting the two points, then $l^2 \in \pp_1^2 \cap \pp_2^2$. Similar
behavior occurs for quartics through five double points:
\[
{4+2 \choose 2} - 5{2+1 \choose 2} = 15-15 = 0 
\]
but the estimate fails: the space of conics has dimension six, and
five points impose at most five conditions, so there is a conic $c \in \cap_{i=1}^5 \pp_i$ and $c^2 \in \cap_{i=1}^5 \pp_i^2$.
\end{exm}
A conjecture of Segre-Harbourne-Gimigliano-Hirschowitz is that this
kind of behavior is the only pathology:
\begin{conj}\label{SHGH}
If a fatpoints ideal $I = \cap \pp_i^{m_i}$ supported at general points 
$p_i$ fails to have the expected Hilbert function $HF(R/I, d)$, then the linear
system $L=dE_0 -\sum m_iE_i$ on the blowup of $\p^2$ at the $p_i$
contains a $-1$ curve $E$ with $E\cdot L \le -2$.
\end{conj}
This statement is opaque, but a nice elementary exposition 
appears in Miranda's paper \cite{miranda}. The takeaway is that 
questions about the Hilbert function of fatpoints on $\p^2$ translate 
into questions about powers of linear forms in three variables, which 
are exactly the ideals $J(v)$ associated to the vertices 
of a tetrahedral complex. In particular, even for generic points 
(hence, for general linear forms), the form of the Hilbert function 
is unknown. Computing the constant term of the Hilbert polynomial
for a generic tetrahedral complex would solve Conjecture~\ref{SHGH}. 
As long as there are eight
or fewer planes, the corresponding linear system is anticanonical,
and all is well. But even for $r=2$ it is possible to have 
nine distinct planes meeting at a point $v$, so $J(v) = \langle l_1^3, \ldots, l_9^3\rangle$. Since $\dim \langle x,y,z\rangle^3_3 = 10$, $J(v) \ne \langle x,y,z\rangle^3$, and also $J(v)$ does not correspond to an anticanonical linear system, so there is no easy way to compute the dimension.

\section{Open questions}  
We close with a number of open questions. The most well known 
open conjecture is the dimension of $S^1_3(\Delta)$ when $\Delta$ is
planar, and the generalization of this as Conjecture~\ref{2r1}. We mention some additional interesting questions.
\subsection{Higher dimensions}
Reconcile the results of the last section with the results of 
Alfeld-Schumaker-Whiteley \cite{asw} on the dimension of $S^1_d(\Delta)$ for 
generic tetrahedral complexes and $d \ge 8$.  Since $r=1$, as soon as there are six or more distinct planes adjacent to each vertex $v$, $J(v)=(x,y,z)^2$; this 
is analogous to the fact that in the $r=1$ planar case, unless there are
only two slopes at $v$, then $J(v) = (x,y)^2$. The result of \cite{asw} 
on $S^1_d(\Delta)$ for $d \ge 8$ is equivalent to the 
vanishing of $H_2(\RJ)$ and $H_1(\RJ)$ in degrees $\ge 7$. 
It seems possible that in {\em any} dimension, if $r=1$ and 
$\Delta$ is generic, then $H_i(\RJ)_d = 0$ for $d \gg 0$ and $i<k$. 
Proving this would yield a combinatorial formula for the dimension of $S^1_d(\Delta)$ for $d \gg 0$. Alfeld~\cite{a} has relevant results for the general case,
and Alfeld-Schumaker~\cite{as3} have results for $k = 3$.
 
A second interesting question here is if there are higher dimensional
analogs of the ``crosscut'' partition found in \cite{cw}, and the 
pseudoboundary partitions studied in \cite{ss97b}. This would give
special classes of subdivision where $H_i(\RJ)_d = 0$ for all $d>i$,
and so by Theorem~\ref{Free}, all dimension computations 
come down to understanding the $R/J(\tau)$ for all faces. As we saw in
\S 5, for $k \ge 3$ this is nontrivial. More generally, find 
formulas for special configurations, as in \cite{s3} and \cite{ssor}.
\subsection{Polyhedral complexes}
In \cite{mdp3}, Dipasquale proves that for a planar polyhedral complex
$P$, if $F=max\{n| \mbox{ there is an n-gon in }P\}$, then 
the formula of Theorem~\ref{PlanarMain} applies if $k \ge (2F-1)(r+1)-1$, 
and makes the 
\begin{conj}
 Theorem~\ref{PlanarMain} applies if $k \ge (F-1)(r+1)-1$.
\end{conj}
This agrees with Conjecture~\ref{2r1} when $P$ is simplicial. 
In higher dimensions, an appropriate
analog \cite{s2} of Theorem~\ref{PlanarMain} gives the top three coefficients
of the Hilbert polynomial, but just as in the simplicial case, the 
remaining coefficients will be extremely delicate. 
\subsection{Supersmoothness}
There has been no attempt to use algebraic methods to study splines
with supersmooth conditions, despite the fact that the supersmooth conditions
may be encoded algebraically. So this area is ripe for exploration.
\vskip .1in
\noindent {\bf Acknowledgments}:  I thank the Mathematisches Forschungsinstitut Oberwolfach, my fellow organizers
Larry Schumaker and Tanya Sorokina, and the participants for a
wonderful and stimulating workshop, Lou Billera for introducing me to the
topic, and Mike Stillman for some of the best fun of my mathematical life. 
{\tt Macaulay2} \cite{danmike} computations were essential to this work. 
\renewcommand{\baselinestretch}{1.0}
\small\normalsize 

\bibliographystyle{amsalpha}

\end{document}